\documentclass[preprint,12pt]{elsarticle}

\usepackage{amssymb}
\usepackage{graphicx}
\usepackage{mathptmx}      

\usepackage{amssymb,amsfonts,amsmath,amscd,latexsym}
\usepackage[english]{babel}
\usepackage{geometry}
\usepackage{latexsym}
\usepackage{amssymb}
\usepackage[all]{xy}
\usepackage{amssymb}
\usepackage{latexsym}
\usepackage{graphicx}
\usepackage[active]{srcltx}
\usepackage{srcltx}
\usepackage{xypic}
\usepackage{amsthm}

\newdefinition{defn}{Definition}[section]
\newdefinition{nim}[defn]{}
\newdefinition{rem}[defn]{Remark}
\newdefinition{ex}[defn]{Example}
\newtheorem{thm}{Theorem}[section]

\newcommand{\Hom}{\mathrm{Hom}}

\newcommand{\Ext}{\mathrm{Ext}}

\newcommand{\Aut}{\mathrm{Aut}}

\newcommand{\HH}{\mathrm{HH}}
\newcommand{\Hc}{\mathrm{H}}

\newcommand{\Ker}{\mathrm{Ker}}
\newcommand{\Ima}{\mathrm{Im}}

\newcommand{\res}{\mathrm{res}}
\newcommand{\Br}{\mathrm{Br}}

\begin{document}

\begin{frontmatter}



\title{On cohomology of saturated fusion systems and support varieties}


\author{Constantin-Cosmin Todea}

\address{Department of Mathematics, Technical University of Cluj-Napoca, Str. G. Baritiu 25,
 Cluj-Napoca 400027, Romania}

\ead{Constantin.Todea@math.utcluj.ro}
\begin{abstract} In this short note we study the cohomology algebra of saturated fusion systems using  finite groups which realize saturated fusion systems and  Hochschild cohomology of group algebras. A similar result to a theorem of Alperin, \cite{AlpMis} is proved for varieties of cohomology algebras of fusions systems associated to block algebras of finite groups.
\end{abstract}


\end{frontmatter}


\section{Introduction}\label{sec1}
A saturated fusion system $\mathcal{F}$ on a finite $p$-group $P$ is a category whose objects are the subgroups of $P$ and whose morphisms satisfy certain axioms mimicking the behavior of a finite group $G$ having $P$ as a Sylow subgroup. The axioms of saturated fusion systems were invented by Puig in early 1990's. The cohomology algebra of a $p$-local finite group with coefficients in $\mathbb{F}_p$ is introduced in \cite[$\S 5$]{BLO}. Let $k$ be an algebraically closed field of characteristic $p$. We denote by $\HH^*(kG)$ the Hochschild cohomology algebra of the group algebra $kG$ and by $\Hc^*(G,k)$ the cohomology algebra of the group $G$ with trivial coefficients.  As in \cite{LiTr} we will use the language of homotopy classes of chain maps (see \cite[2.8, 4.2]{LiTr}). We denote by $\Hc^*(\mathcal{F})$ the algebra of stable elements of $\mathcal{F}$, i.e. the cohomology algebra of the saturated fusion system $\mathcal{F}$, which is the subalgebra of $\Hc^*(P,k)$ consisting of elements $[\zeta]\in \Hc^*(P,k)$ such that
$$\res^P_Q([\zeta])=\res_{\varphi}([\zeta]),$$
for any $\varphi\in\Hom_{\mathcal{F}}(Q,P)$ and any subgroup $Q$ of $P$.
This is the main object of study in this paper. Moreover Broto, Levi and Oliver showed that any saturated fusion system $\mathcal{F}$ has  a non-unique $P-P$-biset $X$ with certain properties formulated by Linckelmann and Webb (see \cite[Proposition 5.5]{BLO}). Such a $P-P$-biset $X$ is called a characteristic biset. Using this biset, S. Park noticed in \cite{Par} a result which says that a saturated fusion system can be realized by a finite group. This finite group is $G=\Aut(X_P)$, that is the group of bijections of the characteristic biset $X$, preserving the right $P$-action. So, by \cite[Theorem 3]{Par}, we can identify $\mathcal{F}$ with $\mathcal{F}_P(G)$ which is the fusion system on $P$ such that for every $Q,R\leq P$ we have
$$\Hom_{\mathcal{F}_P(G)}(Q,R)=\{\varphi:Q\rightarrow R\mid \exists x\in G~ s.t.~ \varphi(u)=xux^{-1},\forall u\in Q\}.$$

In Section \ref{sec2} for a finite group $G$ and $P$ a $p$-subgroup such that  $\mathcal{F}_P(G)$ is a saturated fusion system  we associate and analyze a restriction map from the cohomology algebra of the group $G$ with coefficients in the field $k$ to the cohomology algebra of the fusion system, $\Hc^*(\mathcal{F}_P(G))$. For shortness we denote $\mathcal{F}=\mathcal{F}_P(G)$ and this map by $\rho_{\mathcal{F},G}$. From the above notice that any saturated fusion system $\mathcal{F}$ can be identified with a saturated fusion system of the form $\mathcal{F}_P(G)$, where $G=\Aut(X_P)$ for $X$ a characteristic biset; but we prefer to work under a more general setup.  Although  some results  are straightforward translations of results in the literature, to prove them we need the machinery of transfer maps between Hochschild cohomology algebras of group algebras developed in \cite{LiTr}. Since $\Hc^*(\mathcal{F})$ is a graded commutative finitely generated $k$-algebra we can associate the spectrum of maximal ideals, i.e. the variety denoted $V_{\mathcal{F}}$. For $U$ a finitely generated $kP$-module we define also a support variety, denoted $V_{\mathcal{F}}(U)$, in a similar way  as the usual support variety in group cohomology. The main result of this section is Theorem \ref{thmmain}, where we prove that $\rho_{\mathcal{F},G}$ induces a finite map on varieties and that we can recover $V_{\mathcal{F}}(U)$ from the support variety $V_G(U)$ associated to any finitely generated $kG$-module $U,$ which was first introduced by Carlson in \cite{Carl}.

In Section \ref{sec4} we consider only fusion systems associated to blocks. If $b$ is a block idempotent of the group algebra $kG$ (i.e. a primitive idempotent in the center of the group algebra) with defect group $P$ then it is well known that the set of $b$-Brauer pairs  form a $G$-poset which defines a saturated fusion system. We denote this category by $\mathcal{F}_{G,b}$. By considering a subgroup $H$ in $G$ and a block $c$ of $kH$ with defect group $Q\leq P$ such that $\mathcal{F}_{H,c}$ is a fusion subsystem of $\mathcal{F}_{G,b}$ we can define a restriction map from the cohomology of $\mathcal{F}_{G,b}$ to the cohomology of $\mathcal{F}_{H,c}$, denoted $\res_{b,c}$. In the main theorem of this short note we find necessary and sufficient conditions for this restriction to induce a bijective map from $V_{\mathcal{F}_{G,b}}$ to $V_{\mathcal{F}_{H,c}}$.

For definitions, notations and basic results from block theory  and fusion system theory we follow \cite{The} and \cite{AsKeOl}.
\section{Cohomology of a finite group realizing a saturated fusion system, restrictions and varieties}\label{sec2}

Let $G$ be a finite group and  $P$ be a $p$-subgroup of $G$. We denote by $M$ the $k$-module $kG$ regarded as $kG-kP$-bimodule. In \cite{LiTr} M. Linckelmann defines transfer maps between Hochschild cohomology algebras of symmetric algebras and establishes the compatibility of restriction and transfer maps for group cohomology with these transfer maps for group algebras through some embedding, which we call diagonal induction. Recall the notations of these well-known injective homomorphisms of algebras
$$\delta_G:\Hc^*(G,k)\rightarrow \HH^*(kG), ~\delta_P:\Hc^*(P,k)\rightarrow \HH^*(kP).$$
Moreover from \cite[Propostion 4.8]{LiTr} we have
$\Ima \delta_G\subseteq \HH^*_M(kG),$ where $\HH^*_M(kG)$ is the subalgebra of $M$-stable elements in $\HH^*(kG)$. See \cite[Definition 3.1]{LiTr} for the general case and for the definition of projective elements. By \cite[Example 3.9]{LiTr} we know that the projective elements are in this case $\pi_{M^*}=1_{kP}$, which is invertible in $Z(kP)$ and $\pi_M=[G:P]1_{kG}$. It follows that there is a normalized transfer map $T_{M^*}$ which is equal to $t_{M^*}$.
Let $U$ be a finitely generated $kG$-module, where $G$ is a finite group. The cohomological variety of $U$ (denoted $V_G(U)$), introduced by Carlson in \cite{Carl}, is defined to be the subvariety of the maximal ideal spectrum of $\Hc^*(G,k)$ (denoted $V_G$) determined by $I^*_G(U)$, that is the variety of the quotient algebra $\Hc^*(G,k)/I_G^*(U)$. Here $I_G^*(U)$ is the kernel of the graded $k$-algebra homomorphism $\Hc^*(G,k)\rightarrow\Ext^*_{kG}(U,U)$ induced by the functor $-\otimes_kU$. For more details see \cite{BenII} or \cite{CaTobook}. Moreover, M. Linckelmann noticed in \cite[2.9]{LiVa} that the ideal $I_G^*(U)$ is the kernel of the composition
$$\xymatrix{\Hc^*(G,k)\ar[rr]^{\delta_G}
&&\HH^*(kG)\ar[rr]^{\alpha_U}&&\Ext^*_{kG}(U,U)},$$
where $\alpha_U$ is induced by the functor $-\otimes_{kG}U$.
We denote by $V_{\mathcal{F}}$ the variety of the graded $k$-algebra $\Hc^*(\mathcal{F})$. Let $I^*_{\mathcal{F}}(U)$ be the kernel of the composition of graded $k$-algebra homomorphisms
$$\xymatrix{\Hc^*(\mathcal{F})\ar[rr]^{\delta_P}
&&\HH^*(kP)\ar[rr]^{\alpha_U}&&\Ext^*_{kP}(U,U)},$$
where in the right side $U$ is regarded as $kP$-module by restriction. We define $V_{\mathcal{F}}(U)$ as the subvariety of $V_{\mathcal{F}}$ which consists of all maximal ideals containing $I_{\mathcal{F}}^*(U)$. For $E$ a subgroup of $P$ we denote by $r_{P,E}^*:V_E\rightarrow V_{\mathcal{F}}$ the map between varieties induced by the composition
$$\xymatrix{H^*(\mathcal{F})\ar@{^{(}->}[rr]^{i_P}
&&\Hc^*(P,k)\ar[rr]^{\res^P_E}&&\Hc^*(E,k)}.$$
We need to state the following remark which includes well-known results.
\begin{rem}\label{PropBLO} Let $P$ be a $p$-subgroup of a finite group $G$ such that $\mathcal{F}=\mathcal{F}_P(G)$ is a saturated fusion system on $P$. Let $U$ be a finitely generated $kG$-module. Then $V_{\mathcal{F}}(U)=i_P^*(V_P(U))$ where  $i_P^*:V_P\rightarrow V_{\mathcal{F}}$ is the induced finite surjective map. Moreover, by applying \cite[Theorem 9.4.3]{CaTobook} we have that $V_{\mathcal{F}}(U)=\bigcup_{E}r_{P,E}^*(V_E(U))$, where $E$ runs over the set of elementary abelian subgroups of $P$
\end{rem}
For $P$ a $p$-subgroup of a finite group $G$ such that $\mathcal{F}=\mathcal{F}_P(G)$ is a saturated fusion system on $P$ we denote by
$\rho_{\mathcal{F},G}:\Hc^*(G,k)\rightarrow\Hc^*(\mathcal{F}),$ the homomorphism of graded $k$-algebras induced by restriction.
\begin{thm}\label{thmmain} Let $P$ be a $p$-subgroup of a finite group $G$ such that $\mathcal{F}=\mathcal{F}_P(G)$ is a saturated fusion system on $P$.
\begin{itemize}
\item[(i)] The homomorphism of $k$-algebras $\rho_{\mathcal{F},G}$ induces a finite map $\rho_{\mathcal{F},G}^*:V_{\mathcal{F}}\rightarrow V_G.$
\item[(ii)]  For any finitely generated $kG$-module $U$ we have $$V_{\mathcal{F}}(U)=(\rho_{\mathcal{F},G}^*)^{-1}(V_G(U)).$$
\end{itemize}
\end{thm}
\begin{proof}

Statement (i) is straightforward to check.
Composing the commutative diagram from \cite[Theorem 5.1]{LiVa}, adapted to our situation,  with the obvious commutative diagram (see \cite[Proposition 4.7]{LiTr})
$$\xymatrix{\Hc^*(G,k)\ar[rr]^{\delta_G}\ar[d]^{\rho_{\mathcal{F},G}} && \HH^*_M(kG) \ar[d]^{T_{M^*}} \\
                                                     \Hc^*(\mathcal{F})\ar[rr]^{\delta_P}
                                                     &&\HH^*(kP)
},$$ we obtain the following commutative diagram
$$\xymatrix{\Hc^*(G,k)\ar[rr]^{\alpha_U\circ\delta_G}\ar[d]^{\rho_{\mathcal{F},G}} && \Ext_{kG}^*(U,U) \ar[d]^{\beta_{M^*,U}} \\
                                                     \Hc^*(\mathcal{F})\ar[rr]^{\alpha_{M^*\otimes_{kG} U}\circ\delta_P}
                                                     &&\Ext_{kP}^*(U,U)
},$$
where $\beta_{M^*,U}$ is defined by the functor $M^*\otimes_{kG}-$. It follows  that $\rho_{\mathcal{F},G}(I_G^*(U))\subseteq I_{\mathcal{F}}^*(U)$ (since $M^*\otimes_{kG}U=U$ as $kP$-module), and then $V_{\mathcal{F}}(U)\subseteq(\rho_{\mathcal{F},G}^*)^{-1}(V_G(U)).$
Suppose now that $\alpha\in V_{\mathcal{F}}$ and $\beta\in V_G(U)$ such that $\beta=\rho_{\mathcal{F},G}^*(\alpha)$. We take $U=k$ as trivial $kP$-module in Remark \ref{PropBLO} to obtain that there is an elementary abelian subgroup of $P$ and $\alpha'\in V_E$ such that $\alpha=r_{P,E}^*(\alpha')$. By \cite[Theorem 9.4.3]{CaTobook} there is $E'\leq P$ an elementary abelian subgroup and $\beta'\in V_{E'}(U)$ such that $\beta=(\res^G_{E'})^*(\beta')$. The following equalities hold
$$\rho_{\mathcal{F},G}^*(\alpha)=\rho_{\mathcal{F},G}^*(r_{P,E}^*(\alpha'))=((\res^G_P)^*\circ (\res^P_E\circ i_P)^*)(\alpha')=$$$$(\res^P_E\circ i_P\circ \res^G_P)^*(\alpha')=(\res^G_E)^*(\alpha').$$
We also have $\rho_{\mathcal{F},G}^*(\alpha)=\beta=(\res^G_{E'})^*(\beta')$. It follows that $(\res^G_{E'})^*(\beta')=(\res^G_E)^*(\alpha')$. By \cite[Proposition 9.6.1]{CaTobook} we conclude that $\alpha'\in V_{E'}(U)$, hence $\alpha=r_{P,E}^*(\alpha')\in V_{\mathcal{F}}(U)$ (see Remark \ref{PropBLO}).
\end{proof}
We end this section with the so called Subgroup Theorem (\cite[Theorem 9.6.2]{CaTobook}) for varieties associated to cohomology algebras of saturated fusion systems. Let $\mathcal{F}_1$ be a  saturated fusion system on a $p$-group $P_1$ and $\mathcal{F}_2$ a saturated fusion system on a $p$-group $P_2$ such that $\mathcal{F}_1$ is a fusion subsystem of $\mathcal{F}_2$, where $P_1$ is a subgroup of $P_2$. It is easy to check that $\res^{P_2}_{P_1}(\Hc^*(\mathcal{F}_2))\subseteq \Hc^*(\mathcal{F}_1)$. We denote by $$\res_{\mathcal{F}_2,\mathcal{F}_1}:\Hc^*(\mathcal{F}_2)\rightarrow \Hc^*(\mathcal{F}_1)$$
the homomorphism of $k$-algebras defined by $\res_{\mathcal{F}_2,\mathcal{F}_1}([\zeta])=\res^{P_2}_{P_1}([\zeta])$ for any $[\zeta]\in \Hc^*(\mathcal{F}_2)$.

\begin{thm} With the above assumptions and notations let $U$ be a finitely generated $kP_2$-module. Then
$$V_{\mathcal{F}_1}(U)=(\res_{\mathcal{F}_2,\mathcal{F}_1}^*)^{-1}(V_{\mathcal{F}_2}(U)).$$
\end{thm}
\begin{proof}
Let $X$ be $kP_2$ viewed as $kP_2-kP_1$-bimodule. It is easy to verify the commutativity of  the following diagram
$$\xymatrix{\Hc^*(\mathcal{F}_2)\ar[rr]^{\delta_{P_2}}\ar[d]^{\res_{\mathcal{F}_2,\mathcal{F}_1}} && \HH^*_X(kP_2) \ar[d]^{T_{X^*}} \\
                                                     \Hc^*(\mathcal{F}_1)\ar[rr]^{\delta_{P_1}}
                                                     &&\HH^*_{X^*}(kP_1)
}.$$

We know that $I_{\mathcal{F}_1}(U)=\Ker(\alpha_U\circ\delta_{P_1})$ and that $I_{\mathcal{F}_2}(U)=\Ker(\alpha_U\circ \delta_{P_2})$. By \cite[Theorem 5.1]{LiVa} the following diagram is commutative
$$\xymatrix{\HH^*_X(kP_2)\ar[rr]^{\alpha_U}\ar[d]^{T_{X^*}} && \Ext^*_{kP_2}(U,U) \ar[d]^{\beta_{X^*,U}} \\
                                                     \HH^*_{X^*}(kP_1)\ar[rr]^{\alpha_{X^*\otimes_{kP_1}U}}
                                                     &&\Ext_{kP_1}^*(U,U)
}.$$
The above diagrams assure us the inclusion $\res_{\mathcal{F}_2,\mathcal{F}_1}^*(V_{\mathcal{F}_1}(U))\subseteq V_{\mathcal{F}_2}(U).$
For the reverse inclusion we use Remark \ref{PropBLO} and \cite[Proposition 9.6.1]{CaTobook}.
\end{proof}
\begin{rem} The above theorems could have applications in studying when $\rho_{\mathcal{F},G}$ and $\res_{\mathcal{F}_2,\mathcal{F}_1}$ induce bijective maps on varieties, or when are $F$-isomorphisms. A missing ingredient for doing this is the Quillen stratification for fusion systems which is not appeared yet in a suitable form in the literature. We do not pursue this in here. However, for saturated fusion systems associated to blocks we obtain some positive answers; as we will see in Section \ref{sec4}.
\end{rem}
\section{A theorem of Alperin for cohomology of fusion systems associated to blocks}\label{sec4}
In this section we work only with cohomology algebras of saturated fusion systems associated to $p$-blocks. Let $H$ be a finite subgroup of $G$.
Let $b$ be a block of $kG$ with defect pointed group $P_{\gamma}$ and let $c$ be a block of $kH$ with defect group $Q_{\delta}$ such that $Q\leq P$. Let $i\in \gamma$ be a source idempotent. By Brou\'{e} and Puig \cite[Theorem 1.8]{BrPu} for any subgroup $R$ of $P$, there is a unique block
$e_R$ of $C_G(R)$ such that $\Br_R(i)e_R \neq 0$. Then $(R, e_R)$ is a $b$-Brauer pair, and $e_R$ is also the unique block of $C_G(R)$ such that $(R, e_R) \leq (P, e_P)$. We define
the saturated fusion system $\mathcal{F}_{(P,e_P)}(G,b)$ the  category which has as objects the set of subgroups of $P$; for any two subgroups $R,S$ of $P$ the set of morphisms from $R$ to $S$ in  $\mathcal{F}_{(P,e_P)}(G,b)$ is the set of (necessarily injective) group homomorphisms $\varphi:R\rightarrow S$ for which there is an element $x \in G$ satisfying $\varphi(u)=xux^{-1}$ for all $u\in R$ and satisfying $^x(R, e_R)\leq (S, e_S)$. For shortness we will use the notation $\mathcal{F}_{G,b}$ for this category. The category $\mathcal{F}_{G,b}$ is, up to canonical isomorphism of
categories, independent of the choice of the source idempotent $i$. The analogous definition holds for $\mathcal{F}_{H,c}$; here we denote by $(R,f_R)$ an $c$-Brauer pair, for $R\leq Q$. We denote by $\mathcal{E}_P(b)$,(respectively $\mathcal{E}_Q(c)$) the set of all elementary abelian subgroups of $P$, (respectively $Q$). We mention that the cohomology algebra of the block $b$, which is $\Hc^*(\mathcal{F}_{G,b})$, was introduced and analyzed before the cohomology of general fusion systems by M. Linckelmann in some papers (\cite{LiTr}, \cite{LiVa},\cite{LiQUill}) around the beginning of 2000's.

For the rest of the section let $b,c$ be blocks as above such that $\mathcal{F}_{H,c}$ is a fusion subsystem of $\mathcal{F}_{G,b}$. It is easy to check that $\res^P_Q$ induces a well-defined homomorphism of algebras
$$\res_{b,c}:\Hc^*(\mathcal{F}_{G,b})\rightarrow\Hc^*(\mathcal{F}_{H,c}).$$
Obviously $\res_{b,c}$ is the map $\res_{\mathcal{F}_{G,b},\mathcal{F}_{H,c}}$ defined at the end of Section \ref{sec2}.

In the following lines we introduce some notations adapted to our setting, regarding the Quillen stratification of $\Hc^*(\mathcal{F}_{G,b})$ obtained  in \cite{LiQUill}.  Let $E\in \mathcal{E}_P(b)$. The restriction map $r_{P,E}:\Hc^*(\mathcal{F}_{G,b})\rightarrow \Hc^*(E,k)$ induces  a map on varieties, which we denote $r_{E,b}^*:V_E\rightarrow V_{\mathcal{F}_{G,b}}.$
As usual we define the subvariety of $V_E$
$$V_E^+=V_E\setminus\bigcup_{F<E}(\res^E_F)^*(V_F),$$
and denote the subvarieties of $V_{\mathcal{F}_{G,b}}$
$$V_{G,E,b}=r_{E,b}^*(V_E),~~~ V_{G,E,b}^+=r_{E,b}^*(V_E^+).$$
Finally we set $W_G(E)=N_G(E,e_E)/EC_G(E)$, where $(E,e_E)$ is the unique $b$-Brauer pair such that $(E,e_E)\leq (P,e_P)$.
If $E\in\mathcal{E}_Q(c)$ we consider the analogous notations: $$r_{E,c}^*, V_{H,E,c}, V_{H,E,c}^+, W_H(E)=N_G(E,f_E)/EC_H(E).$$ Now the Quillen stratification is $$V_{\mathcal{F}_{G,b}}=\bigcup_{E}^{\circ}V_{G,E,b}^+,$$ where $E$ runs over a set of $\mathcal{F}_{G,b}$-conjugacy classes of elements in $\mathcal{E}_P(b)$.
\begin{defn}\label{defnweaklyelemsubfusion} Let $b,c$ be blocks of $kG$, respectively $kH$ such that $\mathcal{F}_{H,c}$ is a fusion subsystem of $\mathcal{F}_{G,b}$. We say that $\mathcal{F}_{H,c}$ is \emph{weakly elementary embedded} in $\mathcal{F}_{G,b}$ if the following conditions are satisfied:
\begin{itemize}
\item[(1)] Whenever $E\in \mathcal{E}_Q(c)$ then $W_H(E)\cong W_G(E)$;
\item[(2)] Every subgroup from $\mathcal{E}_P(b)$ is $\mathcal{F}_{G,b}$-conjugate to a subgroup from $\mathcal{E}_Q(c)$;
\item[(3)] If two subgroups from $\mathcal{E}_Q(c)$ are $\mathcal{F}_{G,b}$-conjugate then they are $\mathcal{F}_{H,c}$-conjugate.
\end{itemize}
\end{defn}

\begin{thm}\label{thmAlperin} Let $b,c$ be blocks of $kG$, respectively $kH$ such that $\mathcal{F}_{H,c}$ is a fusion subsystem of $\mathcal{F}_{G,b}$. The restriction map $\res_{b,c}:\Hc^*(\mathcal{F}_{G,b})\rightarrow\Hc^*(\mathcal{F}_{H,c})$ induces a bijective map $\res_{b,c}^*:V_{\mathcal{F}_{H,c}}\rightarrow V_{\mathcal{F}_{G,b}}$ if and only if $\mathcal{F}_{H,c}$ is weakly elementary embedded in $\mathcal{F}_{G,b}$.
\end{thm}
\begin{proof}  Suppose that the conditions from Definition \ref{defnweaklyelemsubfusion} are satisfied. For injectivity let $m_1,m_2\in V_{\mathcal{F}_{H,c}}$ such that $\res_{b,c}^*(m_1)=\res_{b,c}^*(m_2)$. By the Quillen stratification for $V_{\mathcal{F}_{H,c}}$ there are two subgroup $E_1,E_2\in \mathcal{E}_Q(c)$ unique up to $\mathcal{F}_{H,c}$-conjugacy, and $\gamma_1\in V_{E_1}^+, \gamma_2\in V_{E_2}^+$ such that
$$m_1=r_{E_1,c}^*(\gamma_1), ~m_2=r_{E_2,c}^*(\gamma_2).$$
It follows that $$(\res_{b,c}^*\circ r_{E_1,c}^*)(\gamma_1)=(\res_{b,c}^*\circ r_{E_2,c}^*)(\gamma_2),$$
 and since
$$r_{E_1,c}\circ \res_{b,c}=r_{E_1,b},~r_{E_2,c}\circ \res_{b,c}=r_{E_2,b},$$
we obtain that  $ r_{E_1,b}^*(\gamma_1)=r_{E_2,b}^*(\gamma_2).$
Quillen stratification for $V_{\mathcal{F}_{G,b}}$ give us that $E_1$ is $\mathcal{F}_{G,b}$-conjugate to $E_2$, hence $E_1$ is $\mathcal{F}_{H,c}$- conjugate to $E_2$ (by  condition (3) of Definition \ref{defnweaklyelemsubfusion}). This allows us to choose $E_1=E_2=E,$ and we have $r_{E,b}^*(\gamma_1)=r_{E,b}^*(\gamma_2)$. The inseparable isogeny from \cite[Theorem 4.2, Proposition 4.3]{LiQUill}, $V_{G,E,b}^+\cong V_E^+/W_G(E)$, is given by $r_{E,b}^*$, hence $\gamma_1, \gamma_2$ are in the same orbit of the action of $W_G(E)$ on $V_E^+$. From Definition \ref{defnweaklyelemsubfusion}, condition (1) we obtain that $\gamma_1, \gamma_2$ are in the same orbit of the action of $W_H(E)$ on $V_E^+$, hence $m_1=r_{E,c}^*(\gamma_1)=r_{E,c}^*(\gamma_2)=m_2.$
Now let $m\in V_{\mathcal{F}_{G,b}}$. There is $F\in\mathcal{E}_P(b)$ and $\gamma\in V_{F}^+$ such that $m=r_{F,b}^*(\gamma)$. From condition (2) of the above definition  we know that there is $F'\in \mathcal{E}_Q(c)$ such that $F$ is $\mathcal{F}_{G,b}$-conjugate to $F'$. Then $r_{F,b}^*(\gamma)=r_{F',b}^*(\gamma)$ and $m=r_{F',b}^*(\gamma)=\res_{b,c}^*(r_{F',c}^*(\gamma))$.

Now we assume that $\res_{b,c}^*$ is bijective. Let $E_1, E_2\in\mathcal{E}_Q(c)$ which are $\mathcal{F}_{G,b}$-conjugate. Then $V_{G,E_1,b}^+=V_{G,E_2,b}^+$, that is $r_{E_1,b}^*(V_{E_1}^+)=r_{E_2,b}^*(V_{E_2}^+)$. It follows that
$$\res_{b,c}^*(r_{E_1,c}^*(V_{E_1}^+))=\res_{b,c}^*(r_{E_2,c}^*(V_{E_2}^+)).$$
Since $\res_{b,c}^*$ is injective we obtain that $$V_{H,E_1,c}^+=r_{E_1,c}^*(V_{E_1}^+)=r_{E_2,c}^*(V_{E_2}^+)=V_{H,E_2,c}^+,$$
and then $E_1,E_2$ are $\mathcal{F}_{H,c}$-conjugate, hence (3) is proved.
For (1) let $E\in \mathcal{E}_Q(c)$. We identify $V_{G,E,b}^+$ with the orbits of the action of $W_G(E)$ on $V_E^+$ and the same for $V_{H,E,c}^+$.
Since $\res_{b,c}^*$ is bijective and
$$\res_{b,c}^*(V_{H,E,c}^+)=\res_{b,c}^*(r_{E,c}^*(V_E^+))=r_{E,b}^*(V_E^+)=V_{G,E,b}^+,$$
we obtain that $\res_{b,c}^*$ induces a bijection from $V_{H,E,c}^+$ to $V_{G,E,b}^+$. This bijection, the free actions of $W_G(E),W_H(E)$ on $V_E^+$ and the existence of an embedding from $W_H(E)$ into $W_G(E)$ assure  the validity of condition (1); how $\res_{b,c}^*$ is defined is an important ingredient for proving isomorphism (1).
For (2) we consider $n\in V_{\mathcal{F}_{H,c}}$ and $m\in V_{\mathcal{F}_{G,b}}$ with vertex $E\in\mathcal{E}_P(b)$ and a source $\gamma\in V_E^+$ such that $\res_{b,c}^*(n)=m$. We get that a $\mathcal{F}_{G,b}$-conjugate of $E$ must be a vertex of $n$, hence in $\mathcal{E}_Q(c)$.
\end{proof}

\begin{rem} With the hypothesis from Theorem \ref{thmAlperin} if $b,c$ are the principal blocks we obtain an analogous statement to \cite[Theorem 2]{AlpMis}.
\end{rem}

\textbf{Acknowledgements.}
The author would like to express his sincere thanks to the referee for his/her comments and suggestions that contributed to shortening and improving previous versions of this note.






\end{document}